\documentclass{amsart}
\usepackage{mathptmx, bbm, amscd, amssymb, enumerate, colonequals, mathdots}

\usepackage{color}
\definecolor{chianti}{rgb}{0.6,0,0}
\definecolor{meretale}{rgb}{0,0,.6}
\definecolor{leaf}{rgb}{0,.35,0}
\usepackage[colorlinks=true, pagebackref, hyperindex, citecolor=meretale, urlcolor=leaf, linkcolor=chianti]{hyperref}

\newtheorem{theorem}{Theorem}[section]
\newtheorem{lemma}[theorem]{Lemma}
\newtheorem{corollary}[theorem]{Corollary}
\newtheorem{proposition}[theorem]{Proposition}

\theoremstyle{definition}

\newtheorem{example}[theorem]{Example}
\newtheorem{remark}[theorem]{Remark}
\numberwithin{equation}{theorem}

\def\ff{\operatorname{frac}}

\def\image{\operatorname{im}}
\def\coker{\operatorname{coker}}

\def\rank{\operatorname{rank}}

\def\Tr{{\operatorname{Tr}}}

\def\GL{\operatorname{GL}}
\def\SL{\operatorname{SL}}

\def\fraka{\mathfrak{a}}

\def\frakm{\mathfrak{m}}
\def\frakn{\mathfrak{n}}
\def\frakp{\mathfrak{p}}
\def\frakq{\mathfrak{q}}

\def\AA{\mathbb{A}}
\def\FF{\mathbb{F}}
\def\NN{\mathbb{N}}
\def\ZZ{\mathbb{Z}}

\def\ge{\geqslant}
\def\le{\leqslant}
\def\phi{\varphi}

\def\to{\longrightarrow}
\def\mapsto{\longmapsto}

\def\Hom{\operatorname{Hom}}
\def\Ext{\operatorname{Ext}}

\begin{document}
\title[Local cohomology of modular invariant rings]{Local cohomology of modular invariant rings}

\author{Kriti Goel}
\address{Department of Mathematics, University of Utah, 155 South 1400 East, Salt Lake City, UT~84112, USA}
\email{kritigoel.maths@gmail.com}

\author{Jack Jeffries}
\address{Department of Mathematics, University of Nebraska-Lincoln, 203 Avery Hall, Lincoln, NE~68588, USA}
\email{jack.jeffries@unl.edu}

\author{Anurag K. Singh}
\address{Department of Mathematics, University of Utah, 155 South 1400 East, Salt Lake City, UT~84112, USA}
\email{singh@math.utah.edu}

\thanks{K.G. was supported by a Fulbright-Nehru Postdoctoral Research Fellowship at the University of Utah, J.J. by NSF CAREER Award DMS~2044833, and A.K.S. by NSF grants DMS~1801285 and DMS~2101671.}


\begin{abstract}
For $K$ a field, consider a finite subgroup $G$ of $\operatorname{GL}_n(K)$ with its natural action on the polynomial ring $R\colonequals K[x_1,\dots,x_n]$. Let $\mathfrak{n}$ denote the homogeneous maximal ideal of the ring of invariants $R^G$. We study how the local cohomology module $H^n_{\mathfrak{n}}(R^G)$ compares with $H^n_{\mathfrak{n}}(R)^G$. Various results on the $a$-invariant and on the Hilbert series of~$H^n_\mathfrak{n}(R^G)$ are obtained as a consequence.
\end{abstract}
\maketitle

\section{Introduction}

Let $K$ be a field. Consider a finite group $G$ acting on a polynomial ring $R\colonequals K[x_1,\dots,x_n]$ via degree-preserving $K$-algebra automorphisms; the action of $G$ on $R$ is completely determined by its action on one-forms, so there is little loss of generality in taking $G$ to be a finite subgroup of $\GL_n(K)$, with the action given by
\[
M\colon X\mapsto MX,
\]
where $X$ is a column vector of the indeterminates; this is the action of $G$ on $R$ considered throughout the present paper. In the \emph{nonmodular} case---when the order of $G$ is invertible in $K$---there is a wealth of results relating properties of the invariant ring~$R^G$ to properties of the group action; several of these fail in the \emph{modular} case, i.e., when the order of $G$ is a multiple of the characteristic of $K$. For instance, in the nonmodular case, the functor $(-)^G$ is exact, yielding an $R^G$-isomorphism of local cohomology modules
\[
H^n_\frakm(R)^G\ \cong\ H^n_\frakn(R^G),
\]
where $\frakm$ and $\frakn$ denote the respective homogeneous maximal ideals of $R$ and $R^G$. This isomorphism no longer holds in the modular case; indeed, one of our goals is to study the failure of this isomorphism. Quite generally, the transfer map provides a surjection~$H^n_\frakm(R)\to H^n_\frakn(R^G)$; when $G$ contains no transvections, we explicitly describe the kernel in Theorem~\ref{theorem:main}. This result may be viewed as a dual formulation of a theorem of Peskin~\cite{Peskin}, that relates the canonical modules of $R$ and of $R^G$, see Remark~\ref{remark:peskin}.

We apply Theorem~\ref{theorem:main} to study the local cohomology $a$-invariant of $R^G$ in~\S\ref{section:a:invariant}, proving that the $a$-invariant of $R^G$ equals that of $R$ if and only if $G$ is a subgroup of the special linear group with no pseudoreflections; see Theorem~\ref{theorem:a:invariant}. In \S\ref{section:hilbert}, we record a surprising consequence of our main theorem towards comparing the ranks of the graded components of the local cohomology modules $H^n_\frakn(R^G)$ and $H^n_\frakm(R)^G$, proving that they coincide when~$G$ is cyclic with no transvections. The study of local cohomology modules of invariant rings of finite groups goes back at least to work of Ellingsrud and Skjelbred \cite{ES}, where they use spectral sequences relating local cohomology and group cohomology to give upper bounds on the depth of modular invariant rings.

The article~\cite{Stanley} by Stanley provides an excellent account of the theory in the nonmodular case; for sources that include the modular case as well, we refer the reader to Benson~\cite{Benson} and Campbell and Wehlau~\cite{Campbell:Wehlau}. We have attempted to keep this paper largely self-contained, and accessible to the reader familiar with the basics of local cohomology; some preliminary results are reviewed or proved in~\S\ref{section:preliminary}, towards simplifying later arguments. Our study is closely related to earlier work on the canonical module and the Gorenstein property of invariant rings, e.g., \cite{Watanabe:1, Watanabe:2, Peskin, Broer, Braun:2011, FW, Hashimoto}; these are discussed briefly in~\S\ref{section:preliminary}.

\section{Preliminary remarks}
\label{section:preliminary}

We begin with some standard facts about finite group actions:

\subsection*{Pseudoreflections}

An element $g\in\GL_n(K)$ of finite order is a \emph{pseudoreflection} if it fixes a hyperplane; by convention, the group identity is not a pseudoreflection. It follows that $g$ is a pseudoreflection precisely if the matrix $g-I$, with $I$ the identity matrix, has rank one. An equivalent formulation is that the Jordan form of $g$, after extending scalars, is 
\[
\left[
\begin{array}{c|cccc}
\zeta & 0 & 0 & \cdots & 0 \\
\hline
0 & 1 & 0 & \cdots & 0 \\
0 & 0 & 1 & & 0 \\
\vdots & \vdots & & \ddots & \\
0 & 0 & 0 & & 1
\end{array}
\right]
\quad\text{ or }\quad
\left[
\begin{array}{cc|ccc}
1 & 1 & 0 & \cdots & 0 \\
0 & 1 & 0 & \cdots & 0 \\
\hline
0 & 0 & 1 & & 0 \\
\vdots & \vdots & & \ddots & \\
0 & 0 & 0 & & 1
\end{array}
\right].
\]
Since $g$ has finite order, the element $\zeta$ in the first case is a root of unity. The second case only occurs when $K$ has characteristic $p>0$; such an element is a \emph{transvection}.

\begin{remark}
\label{remark:1-g}
Fix $g\in G$. We use~$(1-g)R$ to denote the ideal of $R\colonequals K[x_1,\dots,x_n]$ generated by all elements of the form $r-g(r)$ for~$r\in R$. Since
\[
(1-g)(r_1r_2)\ =\ r_2(1-g)(r_1) + g(r_1)(1-g)(r_2),
\] 
the ideal $(1-g)R$ is generated by the elements $(1-g)(x_i)$ for $1\le i\le n$. Note that $g$ is a pseudoreflection if and only if the ideal $(1-g)R$ has height one.
\end{remark}

\subsection*{Transfer}
Let $G$ be a finite subgroup acting on a ring $R$. For a subgroup $H$, the \emph{transfer map} $\Tr_H^G\colon R^H \to R^G$ is defined as
\[
\Tr_H^G(r)\colonequals\!\!\!\!\!\!\sum_{gH\in G/H}\!\!\!\!\!\!g(r),
\]
where the sum is over a set of left coset representatives. It is straightforward to see that~$\Tr_H^G$ is an $R^G$-linear map, independent of the coset representatives. Precomposing with the inclusion $R^G \subseteq R^H$, the composition
\[
\CD
R^G @>>> R^H @>\Tr_H^G>> R^G
\endCD
\]
is multiplication by the integer $[G:H]$, i.e., by the index of $H$ in $G$. It follows that $\Tr_H^G$ is surjective if $[G:H]$ is invertible in $R$.

When $H$ is the subgroup consisting only of the identity element, we use $\Tr^G$ or $\Tr$ to denote the transfer map $R\to R^G$.

The following lemma appears in various forms in the literature, e.g.,~\cite[Theorem~2.4]{Feshbach}, \cite[Proposition~3.7]{Braun:2011}, and~\cite[Theorem~2.4.5]{Neusel:Smith}; we include a self-contained proof:

\begin{lemma}
\label{lemma:transfer}
Let $G$ be a finite subgroup of $\GL_n(K)$, without transvections, acting on the polynomial ring~$R\colonequals K[x_1,\dots,x_n]$. Then the image of the transfer map $\Tr\colon R\to R^G$ is an ideal of~$R^G$ of height at least two.
\end{lemma}

\begin{proof}
The transfer map is surjective in the nonmodular case, so assume that $K$ has positive characteristic $p$. The claim reduces to the case where $K$ is algebraically closed, as we now assume. Let $\frakp$ be a prime ideal of $R^G$ height one, and $\frakq$ a height one prime of $R$ containing $\frakp$. It suffices to show that there is a maximal ideal~$\frakm$ of $R$, containing $\frakq$, such that $\Tr(R)\not\subseteq\frakm$.

By Remark~\ref{remark:1-g}, the prime $\frakq$ does not contain an ideal of the form $(1-g)R$ for any group element $g$ of order $p$, since such an element would then be a transvection. Let $\fraka$ denote the product of the ideals $(1-g)R$, taken over group elements $g$ of order $p$. Then $\fraka\not\subseteq\frakq$, so there exists a point~$(a_1,\dots,a_n)\in\AA^n_K$ that lies in the algebraic set $V(\frakq)$ but not in $V(\fraka)$. Set~$\frakm\colonequals(x_1-a_1,\dots,x_n-a_n)R$. We claim that $g(\frakm)\neq\frakm$ for each $g\in G$ of order $p$.

If the claim is false, there exists an element $g$ of order $p$ such that
\[
g(x_i-a_i)\ =\ g(x_i)-a_i\ \in\ \frakm\qquad\text{ for each $1\le i\le n$}.
\]
But $x_i-a_i\in\frakm$ as well, so $x_i-g(x_i)\in\frakm$ for each $i$. These generate $(1-g)R$, yielding a contradiction. This proves the claim.

Consider the action of $G$ on the set of maximal ideals of $R$. Since the stabilizer $H$ of $\frakm$ has no elements of order~$p$, the order of $H$ is invertible in $K$. The transfer map $R\to R^G$ factors as
\[
\CD
R @>\Tr^H>> R^H @>\Tr_H^G>> R^G,
\endCD
\]
where the first map is surjective, so it suffices to show that the image of $\Tr_H^G$ is not contained in $\frakm$. Let $\{g_1,\dots,g_\ell\}$ be coset representatives for $G/H$, where $g_1H=H$. Then
\[
\frakm=g_1^{-1}(\frakm),\ g_2^{-1}(\frakm),\ \dots,\ g_\ell^{-1}(\frakm)
\]
are distinct maximal ideals of $R$, so there exists an element $r\in R$ with $r\in g_i^{-1}(\frakm)$ for each~$i\le 2\le\ell$, and $r\notin\frakm$. These conditions are preserved when $r$ is replaced by its orbit product under $H$, so we may assume $r\in R^H$. But then
\begin{alignat*}3
\Tr_H^G(r) &= g_1(r)+g_2(r)\dots+g_\ell(r)\\
&\equiv r\mod\frakm.
\end{alignat*}
It follows that $\Tr_H^G(R^H)$ is not contained in $\frakm$.
\end{proof}

\subsection*{Local cohomology and the canonical module}

Let $S$ be an $\NN$-graded ring that is finitely generated over a field $S_0=K$. Let $\frakn$ denote the homogeneous maximal ideal of $S$, and set~$n\colonequals\dim S$. Let $y_1,\dots,y_n$ be a~\emph{homogeneous system of parameters} for $S$, i.e., a sequence of $n$ homogeneous elements that generate an ideal with radical $\frakn$. For an $S$-module~$M$ and an integer $k\ge 0$, the \emph{local cohomology} module $H^k_\frakn(M)$ is defined as
\[
H^k_\frakn(M) = \varinjlim_i\Ext^k_S(S/\frakn^i,M),
\]
and may be identified with the \v Cech cohomology module $\check H^k(y_1,\dots,y_n;\,S)$, i.e., the~$k$-th cohomology of the \v Cech complex
\[
\CD
0 @>>> M @>>> \bigoplus\limits_i M_{y_i} @>>> \bigoplus\limits_{i<j} M_{y_iy_j} @>>> \cdots @>>> M_{y_1\cdots y_n} @>>> 0.
\endCD
\]
In particular, this identifies $H^n_\frakn(M)$ with
\[
\frac{M_{y_1\cdots y_n}}{\sum_i M_{y_1\cdots \hat{y}_i\cdots y_n}}.
\]
Under this identification, a local cohomology class
\[
\left[\frac{m}{y_1^d\cdots y_n^d}\right]\ \in\ H^n_\frakn(M),
\]
for $m\in M$, is zero if and only if there exists an integer $\ell\ge0$ such that
\[
m(y_1\cdots y_n)^\ell\ \in\ \big(y_1^{d+\ell},\ \dots,\ y_n^{d+\ell}\big)M.
\]

When $M$ is a $\ZZ$-graded $S$-module, each $H^k_\frakn(M)$ acquires a natural $\ZZ$-grading. Following Goto and Watanabe~\cite{Goto:Watanabe}, the \emph{$a$-invariant} of the ring $S$, denoted $a(S)$, is the largest integer~$a$ such that the graded component ${[H_\frakn^n(S)]}_a$ is nonzero.

Let $M$ be a $\ZZ$-graded $S$-module. We use $M(i)$ to denote the module with the shifted grading ${[M(i)]}_j={[M]}_{i+j}$ for each $j\in\ZZ$. The \emph{graded $K$-dual} of $M$, denoted $M^*$, is the~$S$-module with graded components
\[
{[M^*]}_i\ =\ \Hom_K(M,K(i)),
\]
where $\Hom_K(M,K(i))$ is the vector space of degree preserving $K$-linear maps $M\to K(i)$. Assume now that $S$ is normal; the \emph{canonical module} of $S$ is 
\[
\omega_S\colonequals H_\frakn^n(S)^*.
\]
When the ring $S$ is Gorenstein, one has a degree-preserving isomorphism
\[
\omega_S\ \cong\ S(a),
\]
where $a=a(S)$. A normal $\NN$-graded ring $S$ is Gorenstein precisely if it is Cohen-Macaulay and $\omega_S$ is a cyclic $S$-module; dropping the Cohen-Macaulay condition, a normal $\NN$-graded ring $S$ is \emph{quasi-Gorenstein} if $\omega_S$ is a cyclic $S$-module.

Let $G$ be a finite subgroup of $\GL_n(K)$, acting on a polynomial ring~$R$. In the nonmodular case, the invariant ring $R^G$ is Cohen-Macaulay by~\cite{Hochster:Eagon}, though it need not be Cohen-Macaulay in the modular case; this leads to interest in the quasi-Gorenstein property. We summarize some of the work in this direction:

Suppose first that the order of $G$ is invertible in the field $K$; this is the nonmodular case. Watanabe proved that if $G\subseteq\SL_n(K)$, then $R^G$ is Gorenstein~\cite{Watanabe:1}, and that if $G$ contains no pseudoreflections, then the converse holds as well, i.e., if $R^G$ is Gorenstein, then $G\subseteq\SL_n(K)$, see~\cite{Watanabe:2}. Braun~\cite{Braun:2011} proved analogues of these in the modular case when $G$ contains no pseudoreflections: the ring $R^G$ is quasi-Gorenstein if and only if~$G$ is contained in $\SL_n(K)$. Some of these results are extended in~\cite{FW} and~\cite{Hashimoto}.

It was conjectured that if $R^G$ is Cohen-Macaulay and $G\subseteq\SL_n(K)$, then $R^G$ is Gorenstein,~\cite[Conjecture~5]{KKMMVW}; while this is true in the nonmodular case by~\cite{Watanabe:1}, the conjecture was shown to be false by Braun~\cite{Braun:2016}, with the simplest example being the subgroup~$G$ of $\SL_2(\FF_9)$ generated by
\[
\begin{bmatrix}
\zeta & 0\\
0 & \zeta^{-1}
\end{bmatrix}
\quad\text{ and }\quad
\begin{bmatrix}
1 & 1\\
0 & 1
\end{bmatrix},
\]
where $\zeta$ is a primitive $4$-th root of unity. Note that $G$ contains a transvection---as it must!

\subsection*{The group action on local cohomology}

Let $G$ be a finite subgroup of $\GL_n(K)$, acting on a polynomial ring~$R\colonequals K[x_1,\dots,x_n]$. The action of $G$ on $H^n_\frakm(R)$ may be interpreted in several equivalent ways: for $g\in G$, the automorphism $g\colon R\to R$ induces a map
\[
\CD
H^n_\frakm(R) @>g>> H^n_{g(\frakm)}(R) = H^n_{\frakm}(R),
\endCD
\]
where the equality is simply because $g(\frakm)=\frakm$.

Alternatively, let $y_1,\dots,y_n$ be a homogeneous system of parameters for $R^G$, and use the identification of $H^n_{\frakm}(R)$ with \v Cech cohomology $\check H^n(y_1,\dots,y_n;\,R)$. Under this identification, for $g\in G$ and $r\in R$ one has
\[
\eta\colonequals\left[\frac{r}{y_1^d\cdots y_n^d}\right]\mapsto\left[\frac{g(r)}{y_1^d\cdots y_n^d}\right]\ =\ g(\eta).
\]
Note that $\eta$ is fixed by $g$ precisely if there exists an integer $\ell\ge0$ such that
\[
\big(g(r)-r\big)(y_1\cdots y_n)^\ell\ \in\ \big(y_1^{d+\ell},\ \dots,\ y_n^{d+\ell}\big)R.
\]
Since $y_1,\dots,y_n$ is a regular sequence on $R$, this is equivalent to
\[
g(r)-r\ \in\ \big(y_1^d,\ \dots,\ y_n^d\big)R.
\]
It follows that $\eta$ as above if fixed by $g$ precisely if the image of $r$ in the Artinian ring
\[
A\colonequals R/(y_1^d,\dots,y_n^d)R
\]
is fixed by $g$ under the induced action. More generally, $A$ is isomorphic as a $G$-module to the submodule of $H^n_{\frakm}(R)$ consisting of elements of the form
\[
\left[\frac{r}{y_1^d\cdots y_n^d}\right],\quad\text{ for $r\in R$}.
\]
Yet another point of view may be obtained from the ideas surrounding Remark~\ref{remark:action}; we leave this to the interested reader.

Recall that the transfer map $\Tr\colon R\to R^G$ is a homomorphism of $R^G$-modules, and hence induces a map
\begin{equation}
\label{equation:transfer:lc}
\CD
H^n_\frakn(R) @>\Tr>> H^n_\frakn(R^G),
\endCD
\end{equation}
where $\frakn$ is the homogeneous maximal ideal of $R^G$. Since $\frakn R$ has radical $\frakm$, one may identify the modules $H^n_\frakn(R)$ and $H^n_\frakm(R)$. The transfer map~\eqref{equation:transfer:lc} is then precisely the map~$H^n_\frakm(R) \to H^n_\frakn(R^G)$ with
\[
\left[\frac{r}{y_1^d\cdots y_n^d}\right]\mapsto\left[\frac{\Tr(r)}{y_1^d\cdots y_n^d}\right],
\]
where $r\in R$, and $y_1,\dots,y_n$ is a homogeneous system of parameters for $R^G$, as above.

\subsection*{Maps on local cohomology}

For a local ring $(S,\frakn)$, and $M$ a finitely generated $S$-module, the local cohomology modules $H^k_\frakn(M)$ vanish for $k>\dim M$. It follows that the functor~$H^{\dim S}_\frakn(-)$ is right-exact. More generally:

\begin{lemma}
\label{lemma:locally:iso}
Let $(S,\frakn)$ be a local ring and set $n\colonequals\dim S$. Let
\[
\CD
A @>\alpha>> B @>\beta>> C @>>> 0
\endCD
\]
be a complex of finitely generated $S$-modules.
\begin{enumerate}[\quad\rm(1)]
\item\label{lemma:locally:iso:1}
If $B_\frakp\to C_\frakp$ is surjective for each prime ideal $\frakp$ with $\dim S/\frakp=n$, then the induced map $H^n_\frakn(B)\to H^n_\frakn(C)$ is surjective.

\item\label{lemma:locally:iso:2}
If $B_\frakp\to C_\frakp$ is injective for each prime ideal $\frakp$ with $\dim S/\frakp=n$, and surjective for each $\frakp$ with $\dim S/\frakp=n-1$, then $H^n_\frakn(B)\to H^n_\frakn(C)$ is an isomorphism.

\item\label{lemma:locally:iso:3}
If $B_\frakp\to C_\frakp$ is surjective for each $\frakp$ with $\dim S/\frakp=n-1$, and $A_\frakp\to B_\frakp\to C_\frakp$ is exact for each $\frakp$ with $\dim S/\frakp=n$, then the induced sequence
\[
\CD
H^n_\frakn(A) @>>> H^n_\frakn(B) @>>> H^n_\frakn(C) @>>> 0
\endCD
\]
is exact.
\end{enumerate}
\end{lemma}

\begin{proof}
The exact sequence $B \to C \to \coker\beta \to 0$ induces
\[
\CD
H^n_\frakn(B) @>>> H^n_\frakn(C) @>>> H^n_\frakn(\coker\beta) @>>> 0.
\endCD
\]
Since $(\coker\beta)_\frakp$ vanishes for each prime $\frakp$ with $\dim S/\frakp=n$, one has $\dim(\coker\beta)<n$. But then $H^n_\frakn(\coker\beta)=0$, proving~\eqref{lemma:locally:iso:1}.

For~\eqref{lemma:locally:iso:2}, consider the exact sequences
\[
\CD
0 @>>> \ker\beta @>>> B @>>> \image\beta @>>> 0
\endCD
\]
and
\[
\CD
0 @>>> \image\beta @>>> C @>>> \coker\beta @>>> 0.
\endCD
\]
The hypothesis $(\ker\beta)_\frakp=0$ for each $\frakp$ with $\dim S/\frakp=n$ implies that $\dim(\ker\beta)<n$, so~$H^n_\frakn(\ker\beta)=0$. Using the first sequence, $H^n_\frakn(B)\to H^n_\frakn(\image\beta)$ is an isomorphism.

Similarly, since $(\coker\beta)_\frakp=0$ for each prime $\frakp$ with $\dim S/\frakp=n-1$, it follows that~$\dim(\coker\beta)<n-1$, so $H^{n-1}_\frakn(\coker\beta)=0=H^n_\frakn(\coker\beta)$. Passing to local cohomology, the second displayed sequence yields the isomorphism $H^n_\frakn(\image\beta)\to H^n_\frakn(C)$.

For~\eqref{lemma:locally:iso:3}, we may replace $A$ by its image in $B$, and then apply~\eqref{lemma:locally:iso:2} to $B/A\to C$ to obtain the isomorphism $H^n_\frakn(B/A)\to H^n_\frakn(C)$. Combine this with the exact sequence
\[
\CD
H^n_\frakn(A) @>>> H^n_\frakn(B) @>>> H^n_\frakn(B/A) @>>> 0.
\endCD
\qedhere
\]
\end{proof}

\section{Comparing local cohomology}

\begin{theorem}
\label{theorem:main}
For $K$ a field, let $G$ be a finite subgroup of $\GL_n(K)$, without transvections, acting on the polynomial ring~$R\colonequals K[x_1,\dots,x_n]$. Then there is an exact sequence
\[
\CD
\bigoplus\limits_{g \in G} H^n_\frakm(R) @>\alpha>> H^n_\frakm(R) @>\Tr>> H^n_\frakn(R^G) @>>> 0,
\endCD
\]
where $\frakm$ and $\frakn$ denote the respective homogeneous maximal ideals of $R$ and $R^G$, and
\[
\alpha\colon(\eta_g)_{g\in G} \mapsto \sum_{g\in G}\big(\eta_g - g(\eta_g)\big).
\]
\end{theorem}

\begin{proof}
Note that the ideal $\frakn R$ is $\frakm$-primary, so $H^n_\frakm(R)=H^n_\frakn(R)$. In view of Lemma~\ref{lemma:locally:iso}\,\eqref{lemma:locally:iso:3}, it suffices to consider the complex of $R^G$-modules
\begin{equation}
\label{equation:nbt:ring}
\CD
\bigoplus\limits_{g \in G} R @>\alpha>> R @>\Tr>> R^G @>>> 0,
\endCD
\end{equation}
where
\[
\alpha\colon(r_g)_{g\in G} \mapsto \sum_{g\in G}\big(r_g - g(r_g)\big),
\]
and verify that $\Tr\colon R\to R^G$ is surjective after localizing at each height one prime $\frakp$ of~$R^G$, and that the sequence~\eqref{equation:nbt:ring} is exact upon tensoring with the fraction field of $R^G$. The surjectivity of $\Tr\colon R\to R^G$ at height one primes comes from Lemma~\ref{lemma:transfer}. For the second verification, let $L$ denote the fraction field of $R$, in which case~$L^G=\ff(R^G)$ as $G$ is finite. We then need to verify the exactness of the sequence
\begin{equation}
\label{equation:nbt:field}
\CD
\bigoplus\limits_{g \in G} L @>\alpha>> L @>\Tr>> L^G @>>> 0.
\endCD
\end{equation}
But $\Tr\colon L\to L^G$ is a surjective map of $L^G$-vector spaces, so its kernel is an $L^G$-vector space of rank $|G|-1$. By the normal basis theorem, there exists $\lambda\in L$ such that
\[
\{g(\lambda)\mid g\in G\}
\]
is an $L^G$-basis for $L$. But then the image of $\alpha$ in~\eqref{equation:nbt:field} contains the $|G|-1$ linearly independent elements $\lambda-g(\lambda)$, as $g$ varies over the nonidentity elements of $G$.
\end{proof}

\begin{remark}
\label{remark:peskin}
Theorem~\ref{theorem:main} admits a dual formulation that extends \cite[Theorem~2.7]{Peskin} as follows. Suppose~$G$ contains no transvections. Using $(-)^*$ for the graded $K$-dual, one has~$H^n_\frakm(R)^*=\omega_R$ and $H^n_\frakn(R^G)^*=\omega_{R^G}$ so Theorem~\ref{theorem:main} yields the exact sequence
\[
\CD
0 @>>> \omega_{R^G} @>>> \omega_R  @>{\alpha^*}>> \bigoplus\limits_{g \in G} \omega_R.
\endCD
\]
Endowing $\omega_R$ with the $G$-action induced by the identification $\omega_R=H^n_\frakm(R)^*$, one has $\ker\alpha^* \cong (\omega_R)^G$. The exact sequence above then gives
\[
\omega_{R^G}\ \cong\ (\omega_R)^G.
\]
This does not require the hypothesis that $R^G$ is Cohen-Macaulay, assumed in \cite{Peskin}.
\end{remark}

\begin{remark}
\label{remark:generators}
In the statement of Theorem~\ref{theorem:main}, one may replace $\bigoplus\limits_{g \in G} H^n_{\frakm}(R)$ by the direct sum over a generating set for $G$, and $\alpha$ by its restriction: if $g,h\in G$, then
\[
(1-hg)(\eta)\ =\ (1-g)(\eta) + (1-h)(g(\eta)).
\]
\end{remark}

The hypothesis that $G$ does not contain transvections is indeed required in Theorem~\ref{theorem:main}:

\begin{example}
Consider the symmetric group $S_2=\langle g\rangle$ acting on $R\colonequals K[x,y]$ by permuting the variables. Then $R^{S_2}=K[e_1,e_2]$, where $e_1\colonequals x+y$ and $e_2\colonequals xy$. While $g$ is a pseudo\-reflection independent of the characteristic of $K$, it is a transvection if and only if $K$ has characteristic two. We examine the complex
\begin{equation}
\label{equation:s2}
\CD
H^2_\frakm(R) @>1-g>> H^2_\frakm(R) @>\Tr>> H^2_\frakn(R^{S_2}) @>>> 0
\endCD
\end{equation}
in degree $-2$. Note that ${[H^2_\frakn(R^{S_2})]}_{-2}=0$, while ${[H^2_\frakm(R)]}_{-2}$, computed via the \v Cech complex on $e_1,e_2$, is the rank one $K$-vector space spanned by
\[
\eta\colonequals\left[\frac{x}{e_1e_2}\right].
\]
Since
\[
(1-g)(\eta)\ =\ \left[\frac{x-y}{e_1e_2}\right]\ =\ \left[\frac{2x}{e_1e_2}\right]\ =\ 2\eta,
\]
the degree $-2$ strand of~\eqref{equation:s2} takes the form
\[
\CD
K @>2>> K @>>> 0 @>>> 0,
\endCD
\]
which is exact precisely when the characteristic of $K$ is other than two, i.e., precisely when the group contains no transvections.
\end{example}

\section{When is the $a$-invariant invariant?}
\label{section:a:invariant}

We record in this section when the $a$-invariant of a ring of invariants coincides with that of the ambient polynomial ring. The following proposition is likely well-known to experts, for example, it is an extension of \cite[Lemma~2.17]{Jeffries:thesis}; see also~\cite[Theorem~1.1]{KPU}.

\begin{proposition}
\label{prop:inequality}
Let $G$ be a finite subgroup of $\GL_n(K)$, acting on a polynomial ring~$R$. Then, for each subgroup $H$ of $G$, one has $a(R^G) \le a(R^H)$.
\end{proposition}

\begin{proof}
Consider the transfer map $\Tr_H^G\colon R^H \to R^G$ given by
\begin{equation}
\label{equation:transfer}
\Tr_H^G(r)\colonequals\!\!\!\!\!\!\sum_{gH\in G/H}\!\!\!\!\!\!g(r).
\end{equation}
Let $L$ denote the fraction field of $R$. Since $G$ and $H$ are finite, one has $L^G=\ff(R^G)$ and~$L^H=\ff(R^H)$. Distinct cosets $gH$ induce distinct automorphisms $g\colon L^H\to L^H$, so Dedekind's theorem implies that the corresponding characters ${(L^H)}^\times\to{(L^H)}^\times$ are linearly independent over $L^H$, and hence over $L^G$. It follows that their sum
\[
\sum g\colon {(L^H)}^\times\to L^H,
\]
taken over coset representatives, is a nonzero map, and hence that the transfer map~\eqref{equation:transfer} is nonzero. As the transfer is~$R^G$-linear, one has an exact sequence of $R^G$-modules
\[
\CD
R^H @>{\Tr_H^G}>> R^G @>>> R^G/\image(\Tr^G_H) @>>> 0.
\endCD
\]
Applying the functor $H^n_\frakn(-)$, one obtains the surjection
\[
\CD
H^n_\frakn(R^H) @>{\Tr_H^G}>> H^n_\frakn(R^G),
\endCD
\]
since $R^G/\image(\Tr^G_H)$ has smaller dimension. The homogeneous maximal ideals of $R^H$ and~$R^G$ agree up to radical, so the assertion follows.
\end{proof}

The following is~\cite[Theorem~2.18]{Jeffries:thesis}, and also related to work of Broer~\cite{Broer}:

\begin{corollary}
\label{corollary:jack}
Let $K$ be a field of characteristic $p>0$, and $G$ a finite subgroup of $\GL_n(K)$ acting on a polynomial ring~$R\colonequals K[x_1,\dots,x_n]$. If $a(R^G)=a(R)$, and $p$ divides the order of~$G$, then the inclusion $R^G\subseteq R$ is not $R^G$-split.
\end{corollary}

\begin{proof}
Consider the maps of rank one $K$-vector spaces
\[
\CD
{[H^n_\frakn(R^G)]}_{-n} @>i>> {[H^n_\frakm(R)]}_{-n} @>\Tr>> {[H^n_\frakn(R^G)]}_{-n},
\endCD
\]
where $i$ is induced by the inclusion $R^G\subseteq R$. The composition is then multiplication by~$|G|$, which equals zero in $K$. As $\Tr$ above is surjective, the map $i$ must be zero. But then the inclusion $R^G\subseteq R$ is not $R^G$-split.
\end{proof}

\begin{remark}
\label{remark:action}
Let $G$ be a finite subgroup of $\GL_n(K)$, acting on~$R\colonequals K[x_1,\dots,x_n]$. We claim that for each $g\in G$ and $\eta\in{[H^n_\frakm(R)]}_{-n}$, one has
\[
g\cdot\eta\ =\ {(\det g)}^{-1}\eta.
\]

Since ${[H^n_\frakm(R)]}_{-n}$ has rank one, without loss of generality, take $\eta$ to be
\[
\left[\frac{1}{x_1\cdots\,x_n}\right].
\]
If $f_1,\dots,f_n$ is a homogeneous system of parameters for $R$, the natural isomorphism between \v Cech and local cohomology induces a natural isomorphism between the \v Cech cohomology modules $\check H^n(x_1,\dots,x_n;\,R)$ and $\check H^n(f_1,\dots,f_n;\,R)$. To make this explicit, following~\cite[Theorem~4.18]{Kunz}, let $A$ be a matrix over $R$, such that
\[
\begin{bmatrix}
f_1 \\ \vdots\\ f_n
\end{bmatrix}
\ =\ A
\begin{bmatrix}
x_1 \\ \vdots\\ x_n
\end{bmatrix}.
\]
Then, under the isomorphism $\check H^n(x_1,\dots,x_n;\,R)\to \check H^n(f_1,\dots,f_n;\,R)$, one has 
\[
\left[\frac{1}{x_1\cdots\,x_n}\right] \mapsto \left[\frac{\det A}{f_1\cdots f_n}\right].
\]
It follows that
\[
g\cdot\left[\frac{1}{x_1\cdots\,x_n}\right]\ =\ \left[\frac{1}{g(x_1)\cdots g(x_n)}\right],
\]
viewed as an element of $\check H^n(g(x_1),\dots,g(x_n);\,R)$, corresponds to
\[
\left[\frac{{(\det g)}^{-1}}{x_1\cdots\,x_n}\right]\ =\ {(\det g)}^{-1}\eta
\]
in $\check H^n(x_1,\dots,x_n;\,R)$.
\end{remark}

The following theorem has been obtained independently by Hashimoto~\cite{Hashimoto}:

\begin{theorem}
\label{theorem:a:invariant}
For $K$ a field, let $G$ be a finite subgroup of $\GL_n(K)$ acting on the polynomial ring~$R\colonequals K[x_1,\dots,x_n]$. Then $a(R^G)=a(R)$ if and only if $G$ is a subgroup of~$\SL_n(K)$ that contains no pseudoreflections.
\end{theorem}

\begin{proof}
We first show that if $G$ contains a pseudoreflection, then $a(R^G)<a(R)$. In view of Proposition~\ref{prop:inequality}, it suffices to consider the case where $G$ is a cyclic group, generated by a pseudoreflection $g$. After extending scalars, we may assume that $g$ takes the form
\[
\left[
\begin{array}{c|cccc}
\zeta & 0 & 0 & \cdots & 0 \\
\hline
0 & 1 & 0 & \cdots & 0 \\
0 & 0 & 1 & & 0 \\
\vdots & \vdots & & \ddots & \\
0 & 0 & 0 & & 1
\end{array}
\right]
\quad\text{ or }\quad
\left[
\begin{array}{cc|ccc}
1 & 1 & 0 & \cdots & 0 \\
0 & 1 & 0 & \cdots & 0 \\
\hline
0 & 0 & 1 & & 0 \\
\vdots & \vdots & & \ddots & \\
0 & 0 & 0 & & 1
\end{array}
\right],
\]
where $\zeta$ is a primitive $k$-th root of unity. In the first case, $R^G=K[x_1^k,\, x_2,\dots,x_n]$, and in the second $R^G=K[x_1^p-x_1x_2^{p-1},\, x_2,\dots,x_n]$, where $p>0$ is the characteristic of $K$. In each case $R^G$ is a polynomial ring, with $a(R^G)$ strictly less than $a(R)$.

It remains to verify that if $G$ has no pseudoreflections, then $a(R^G)=a(R)$ if and only if~$G$ is a subgroup of $\SL_n(K)$. The exact sequence from Theorem~\ref{theorem:main}, when restricted to the degree $-n$ strand, gives an exact sequence of $K$-vector spaces
\[
\CD
\bigoplus\limits_{g \in G} {[H^n_\frakm(R)]}_{-n} @>{\alpha}>> {[H^n_\frakm(R)]}_{-n} @>\Tr>> {[H^n_\frakn(R^G)]}_{-n} @>>> 0.
\endCD
\]
Since ${[H^n_\frakm(R)]}_{-n}$ is a rank one vector space, it follows that $a(R^G)=-n$ if and only if the map $\alpha$ above is identically zero, i.e., if and only if the map
\[
\CD
{[H^n_\frakm(R)]}_{-n} @>{1-g}>> {[H^n_\frakm(R)]}_{-n}
\endCD
\]
is zero for each $g\in G$. Taking
\[
\eta\colonequals\left[\frac{1}{x_1\cdots\,x_n}\right]
\]
as in Remark~\ref{remark:action}, this is equivalent to the condition that
\[
\eta-g(\eta)\ =\ \eta-(\det g)^{-1}\eta
\]
is zero for each $g$, i.e., that $\det g=1$ for each $g\in G$.
\end{proof}

\section{Hilbert series of local cohomology}
\label{section:hilbert}

Theorem~\ref{theorem:main} has an amusing consequence for the Hilbert series of local cohomology:

\begin{corollary}
\label{corollary:hilbert}
For $K$ a field, let $G$ be a finite cyclic subgroup of $\GL_n(K)$, without transvections, acting on the polynomial ring~$R\colonequals K[x_1,\dots,x_n]$. Then the Hilbert series of $H^n_\frakn(R^G)$ and $H^n_\frakm(R)^G$ coincide, i.e., for each integer $k$, one has
\[
\rank_K{[H^n_\frakn(R^G)]}_k\ =\ \rank_K{[H^n_\frakm(R)^G]}_k.
\]
\end{corollary}

\begin{proof}
Let $G=\langle g \rangle$. Then, by Theorem~\ref{theorem:main} and Remark~\ref{remark:generators}, one has an exact sequence
\[
\CD
H^n_\frakm(R) @>{1-g}>> H^n_\frakm(R) @>\Tr>> H^n_\frakn(R^G) @>>> 0.
\endCD
\]
But the kernel of the first map is precisely $H^n_{\frakm}(R)^G$, so
\[
\CD
0 @>>> H^n_{\frakm}(R)^G @>>> H^n_\frakm(R) @>{1-g}>> H^n_\frakm(R) @>\Tr>> H^n_\frakn(R^G) @>>> 0
\endCD
\]
is exact. Taking the degree $k$ strand, the alternating sum of the ranks is zero.
\end{proof}

We will see in Example~\ref{example:klein} that the equality of Hilbert series need not hold when~$G$ is not cyclic; however, before that, it is worth emphasizing that both $H^n_\frakn(R^G)$ and $H^n_\frakm(R)^G$ are graded~$R^G$-modules, and Corollary~\ref{corollary:hilbert} says precisely that they are isomorphic as graded~$K$-vector spaces. They need not be isomorphic as $R^G$-modules:

\begin{example}
\label{example:alternating}
Consider the alternating group $A_3$ acting on $R\colonequals\FF_3[x,y,z]$ by permuting the variables. The ring of invariants $R^{A_3}$ is then generated by the elements
\[
e_1\colonequals x+y+z,\quad e_2\colonequals xy+yz+zx,\quad e_3\colonequals xyz,\quad \Delta\colonequals x^2y+y^2z+z^2x.
\]
It follows that $R^{A_3}$ is a hypersurface; the defining equation is readily seen to be
\[
\Delta^2 - e_1e_2\Delta + e_2^3 + e_1^3e_3.
\]
Taking a \v Cech complex on $e_1,e_2,e_3$, the socle of the $R^{A_3}$-module $H^3_\frakn(R^{A_3})$ is the rank one vector space spanned by the cohomology class
\[
\eta\colonequals\left[\frac{\Delta}{e_1e_2e_3}\right].
\]
Note that $\eta$ belongs to the kernel of the natural map $H^3_\frakn(R^{A_3})\to H^3_\frakn(R)$ since $R^{A_3}\to R$ is not $R^{A_3}$-split; alternatively, it is a routine verification that
\[
\Delta\ \in\ (e_1, e_2, e_3)R.
\]
We claim that, in contrast with $H^3_\frakn(R^{A_3})$, the socle of $H^3_\frakn(R)^{A_3}$, as an $R^{A_3}$-module, has larger rank: for this, one may verify that the elements
\[
\left[\frac{x\Delta}{e_1^2e_2e_3}\right],\quad
\left[\frac{\Delta}{e_1^2e_2e_3}\right],\quad
\left[\frac{\Delta}{e_1e_2^2e_3}\right],\quad
\left[\frac{1}{e_1e_2e_3}\right],
\]
are all nonzero in $H^3_\frakn(R)$, that they are $A_3$-invariant, and that they are annihilated by the ideal $(e_1, e_2, e_3, \Delta)R^{A_3}$. Note that they have degrees $-3$, $-4$, $-5$, $-6$ respectively.
\end{example}

The equality of Hilbert series, Corollary~\ref{corollary:hilbert}, fails for an action of the Klein-$4$ group:
 
\begin{example}
\label{example:klein}
The following matrices over $\FF_2$ generate the Klein-$4$ group:
\[
\begin{bmatrix}
1 & 0 & 0\\
0 & 1 & 1\\
0 & 0 & 1
\end{bmatrix}
\quad\text{ and }\quad
\begin{bmatrix}
1 & 0 & 1\\
0 & 1 & 0\\
0 & 0 & 1
\end{bmatrix}.
\]
Each of these is a transvection; the invariant ring for this action of the Klein-$4$ group on~$\FF_2[x,y,z]$ is the polynomial ring
\[
\FF_2[z,\ x^2 + xz,\ y^2 + yz].
\]
The situation is more interesting if we take the $2$-fold diagonal embedding, i.e., if we consider the representation of the Klein-$4$ group, over $\FF_2$, determined by the matrices:
\[
g\colonequals
\left[
\begin{array}{ccc|ccc}
1 & 0 & 0 & 0 & 0 & 0\\
0 & 1 & 1 & 0 & 0 & 0\\
0 & 0 & 1 & 0 & 0 & 0\\
\hline
0 & 0 & 0 & 1 & 0 & 0\\
0 & 0 & 0 & 0 & 1 & 1\\
0 & 0 & 0 & 0 & 0 & 1
\end{array}
\right]
\quad\text{ and }\quad
h\colonequals
\left[
\begin{array}{ccc|ccc}
1 & 0 & 1 & 0 & 0 & 0\\
0 & 1 & 0 & 0 & 0 & 0\\
0 & 0 & 1 & 0 & 0 & 0\\
\hline
0 & 0 & 0 & 1 & 0 & 1\\
0 & 0 & 0 & 0 & 1 & 0\\
0 & 0 & 0 & 0 & 0 & 1
\end{array}
\right].
\]
Under the action of this group $G$ on the polynomial ring $R\colonequals\FF_2[u,v,w,x,y,z]$, the following elements are readily seen to be invariant:
\[
w,\quad
z,\quad
u^2 + uw,\quad
v^2 + vw,\quad
x^2 + xz,\quad
y^2 + yz,\quad
uz + wx,\quad
vz + wy.
\]
Indeed, the invariant ring $R^G$ is generated by these elements, and is a complete intersection ring with defining relations
\[
(uz+wx)^2 + (uz+wx)wz + (u^2+uw)z^2 + (x^2+xz)w^2\phantom{.}
\]
and
\[
(vz+wy)^2 + (vz+wy)wz + (v^2+vw)z^2 + (y^2+yz)w^2.
\]
It follows that $R^G$ has Hilbert series
\[
\frac{(1-t^4)^2}{(1-t)^2(1-t^2)^6}\ =\ \frac{(1+t^2)^2}{(1-t)^2(1-t^2)^4}\ =\ 1+2t+9t^2+\cdots.
\]

Set $\frakn$ to be the ideal of $R^G$ generated by the homogeneous system of parameters
\[
w^2,\quad
z^2,\quad
u^2 + uw,\quad
v^2 + vw,\quad
x^2 + xz,\quad
y^2 + yz.
\]
Since $R^G$ is Gorenstein with $a(R^G)=-6$, the Hilbert series above yields
\[
\rank{[H^6_\frakn(R^G)]}_{-6}\ =\ 1 \qquad\text{and}\qquad \rank{[H^6_\frakn(R^G)]}_{-7}\ =\ 2.
\]
We claim that, on the other hand,
\[
\rank{[H^6_\frakn(R)^G]}_{-7}\ =\ 4.
\]

Consider the Artinian ring $A\colonequals R/\frakn R$; we identify~${[H^6_\frakn(R)]}_{-6}$ with ${[A]}_6$, and ${[H^6_\frakn(R)]}_{-7}$ with ${[A]}_5$ as $G$-modules.

The rank one space ${[A]}_6$ has basis $uvwxyz$, which is fixed by $g$ and $h$, (as it must!) since
\[
g\colon uvwxyz\ \mapsto\ u(v+w)wx(y+z)z\ \equiv\ uvwxyz
\]
in $A$, and
\[
h\colon uvwxyz\ \mapsto\ (u+w)vw(x+z)yz\ \equiv\ uvwxyz.
\]

For ${[A]}_5$, we work with the basis $vwxyz$, $uwxyz$, $uvxyz$, $uvwyz$, $uvwxz$, $uvwxy$. The first of these elements is fixed since 
\[
g\colon vwxyz\ \mapsto\ (v+w)wx(y+z)z\ \equiv\ vwxyz
\]
and
\[
h\colon vwxyz\ \mapsto\ vw(x+z)yz\ \equiv\ vwxyz.
\]
Similar calculations show that $uwxyz$, $uvwyz$, $uvwxz$ are fixed by $g$ and $h$. On the other hand
\[
g\colon uvxyz\ \mapsto\ u(v+w)x(y+z)z\ \equiv\ (uv+uw)xyz
\]
and
\[
g\colon uvwxy\ \mapsto\ u(v+w)wx(y+z)\ \equiv\ uvw(xy+xz),
\]
so $g$ fixes no nonzero $\FF_2$-linear combination of $uvxyz$ and $uvwxy$. It follows that the subspace of ${[A]}_5$ fixed by $G$ has basis $vwxyz$, $uwxyz$, $uvwyz$, $uvwxz$.
\end{example}

\section*{Acknowledgments}

We have benefitted from several examples computed using the computer algebra systems \texttt{Macaulay2}~\cite{Macaulay2} and \texttt{Magma}~\cite{Magma}; the use of these is gratefully acknowledged. We are also deeply grateful to Gregor Kemper for sharing the database~\cite{KKMMVW} and for related discussions, and to Mitsuyasu Hashimoto for valuable discussions. We thank the referees for a careful proofreading, and for useful comments.


\end{document}